\documentclass[12pt]{amsart}
\usepackage{amsmath,amssymb,amsbsy,amsfonts,amsthm,latexsym,
  amsopn,amstext,amsxtra,euscript,amscd,stmaryrd,mathrsfs,
  cite,array,mathtools,enumerate}
\usepackage{url}
 \usepackage{comment}   
\usepackage[colorlinks,linkcolor=blue,anchorcolor=blue,citecolor=blue,backref=page]{hyperref}
\usepackage{color}
\usepackage{graphics,epsfig}
\usepackage{graphicx}
\usepackage{float} 
\usepackage[english]{babel}
\usepackage{mathtools}
\usepackage{todonotes}
\usepackage{url}
\usepackage[colorlinks,linkcolor=blue,anchorcolor=blue,citecolor=blue,backref=page]{hyperref}
\DeclareMathAlphabet{\mathmybb}{U}{bbold}{m}{n}

\usepackage[norefs,nocites]{refcheck}

\hypersetup{breaklinks=true}

\usepackage[norefs,nocites]{refcheck}
\usepackage[english]{babel}
\begin{document}

\newtheorem{thm}{Theorem}
\newtheorem{lem}[thm]{Lemma}
\newtheorem{claim}[thm]{Claim}
\newtheorem{cor}[thm]{Corollary}
\newtheorem{prop}[thm]{Proposition} 
\newtheorem{definition}[thm]{Definition}
\newtheorem{rem}[thm]{Remark} 
\newtheorem{question}[thm]{Open Question}
\newtheorem{conj}[thm]{Conjecture}
\newtheorem{prob}{Problem}

\newtheorem{lemma}[thm]{Lemma}

\newcommand{\GL}{\operatorname{GL}}
\newcommand{\SL}{\operatorname{SL}}
\newcommand{\lcm}{\operatorname{lcm}}
\newcommand{\ord}{\operatorname{ord}}
\newcommand{\Op}{\operatorname{Op}}
\newcommand{\Tr}{\operatorname{Tr}}
\newcommand{\Nm}{\operatorname{Nm}}
\def\Res{\operatorname{Res}}
\def\Orb{\operatorname{Orb}}

\numberwithin{equation}{section}
\numberwithin{thm}{section}
\numberwithin{table}{section}

\def\vol {{\mathrm{vol\,}}}
\def\squareforqed{\hbox{\rlap{$\sqcap$}$\sqcup$}}
\def\qed{\ifmmode\squareforqed\else{\unskip\nobreak\hfil
\penalty50\hskip1em\null\nobreak\hfil\squareforqed
\parfillskip=0pt\finalhyphendemerits=0\endgraf}\fi}

\def \ss{\mathsf{s}} 

\def \balpha{\bm{\alpha}}
\def \bbeta{\bm{\beta}}
\def \bgamma{\bm{\gamma}}
\def \blambda{\bm{\lambda}}
\def \bchi{\bm{\chi}}
\def \bphi{\bm{\varphi}}
\def \bpsi{\bm{\psi}}
\def \bomega{\bm{\omega}}
\def \btheta{\bm{\vartheta}}

\newcommand{\bfxi}{{\boldsymbol{\xi}}}
\newcommand{\bfrho}{{\boldsymbol{\rho}}}

\def\Kab{\sfK_\psi(a,b)}
\def\Kuv{\sfK_\psi(u,v)}
\def\SaUV{\cS_\psi(\balpha;\cU,\cV)}
\def\SaAV{\cS_\psi(\balpha;\cA,\cV)}

\def\SUV{\cS_\psi(\cU,\cV)}
\def\SAB{\cS_\psi(\cA,\cB)}

\def\Kmnp{\sfK_p(m,n)}

\def\KKap{\cH_p(a)}
\def\KKaq{\cH_q(a)}
\def\KKmnp{\cH_p(m,n)}
\def\KKmnq{\cH_q(m,n)}

\def\Klmnp{\sfK_p(\ell, m,n)}
\def\Klmnq{\sfK_q(\ell, m,n)}

\def \SALMNq {\cS_q(\balpha;\cL,\cI,\cJ)}
\def \SALMNp {\cS_p(\balpha;\cL,\cI,\cJ)}

\def \SACXMQX {\fS(\balpha,\bzeta, \bxi; M,Q,X)}

\def\SAMJp{\cS_p(\balpha;\cM,\cJ)}
\def\SAMJq{\cS_q(\balpha;\cM,\cJ)}
\def\SAqMJq{\cS_q(\balpha_q;\cM,\cJ)}
\def\SAJq{\cS_q(\balpha;\cJ)}
\def\SAqJq{\cS_q(\balpha_q;\cJ)}
\def\SAIJp{\cS_p(\balpha;\cI,\cJ)}
\def\SAIJq{\cS_q(\balpha;\cI,\cJ)}

\def\RIJp{\cR_p(\cI,\cJ)}
\def\RIJq{\cR_q(\cI,\cJ)}

\def\TWXJp{\cT_p(\bomega;\cX,\cJ)}
\def\TWXJq{\cT_q(\bomega;\cX,\cJ)}
\def\TWpXJp{\cT_p(\bomega_p;\cX,\cJ)}
\def\TWqXJq{\cT_q(\bomega_q;\cX,\cJ)}
\def\TWJq{\cT_q(\bomega;\cJ)}
\def\TWqJq{\cT_q(\bomega_q;\cJ)}

 \def \xbar{\overline x}
  \def \ybar{\overline y}

\def\cA{{\mathcal A}}
\def\cB{{\mathcal B}}
\def\cC{{\mathcal C}}
\def\cD{{\mathcal D}}
\def\cE{{\mathcal E}}
\def\cF{{\mathcal F}}
\def\cG{{\mathcal G}}
\def\cH{{\mathcal H}}
\def\cI{{\mathcal I}}
\def\cJ{{\mathcal J}}
\def\cK{{\mathcal K}}
\def\cL{{\mathcal L}}
\def\cM{{\mathcal M}}
\def\cN{{\mathcal N}}
\def\cO{{\mathcal O}}
\def\cP{{\mathcal P}}
\def\cQ{{\mathcal Q}}
\def\cR{{\mathcal R}}
\def\cS{{\mathcal S}}
\def\cT{{\mathcal T}}
\def\cU{{\mathcal U}}
\def\cV{{\mathcal V}}
\def\cW{{\mathcal W}}
\def\cX{{\mathcal X}}
\def\cY{{\mathcal Y}}
\def\cZ{{\mathcal Z}}
\def\Ker{{\mathrm{Ker}}}

\def\NmQR{N(m;Q,R)}
\def\VmQR{\cV(m;Q,R)}

\def\Xm{\cX_{p,m}}

\def \A {{\mathbb A}}
\def \B {{\mathbb A}}
\def \C {{\mathbb C}}
\def \F {{\mathbb F}}
\def \G {{\mathbb G}}
\def \L {{\mathbb L}}
\def \K {{\mathbb K}}
\def \N {{\mathbb N}}
\def \PP {{\mathbb P}}
\def \Q {{\mathbb Q}}
\def \R {{\mathbb R}}
\def \Z {{\mathbb Z}}
\def \fS{\mathfrak S}
\def \fB{\mathfrak B}

\def \fM{\mathfrak M}

\def\GL{\operatorname{GL}}
\def\SL{\operatorname{SL}}
\def\PGL{\operatorname{PGL}}
\def\PSL{\operatorname{PSL}}
\def\li{\operatorname{li}}
\def\sym{\operatorname{sym}}

\def\Mob{M{\"o}bius }

\def\fF{\EuScript{F}}
\def\M{\mathsf {M}}
\def\T{\mathsf {T}}

\def\sR{{\mathscr R}}

\def\e{{\mathbf{\,e}}}
\def\ep{{\mathbf{\,e}}_p}
\def\eq{{\mathbf{\,e}}_q}

\def\\{\cr}
\def\({\left(}
\def\){\right)}

\def\<{\left(\!\!\left(}
\def\>{\right)\!\!\right)}
\def\fl#1{\left\lfloor#1\right\rfloor}
\def\rf#1{\left\lceil#1\right\rceil}

\def\Tr{{\mathrm{Tr}}}
\def\Nm{{\mathrm{Nm}}}
\def\Im{{\mathrm{Im}}}

\def \oF {\overline \F}

\newcommand{\pfrac}[2]{{\left(\frac{#1}{#2}\right)}}

\def \Prob{{\mathrm {}}}
\def\e{\mathbf{e}}
\def\ep{{\mathbf{\,e}}_p}
\def\epp{{\mathbf{\,e}}_{p^2}}
\def\em{{\mathbf{\,e}}_m}

\def\vec#1{\mathbf{#1}}
\def \va{\vec{a}}
\def \vb{\vec{b}}
\def \vh{\vec{h}}
\def \vk{\vec{k}}
\def \vs{\vec{s}}
\def \vu{\vec{u}}
\def \vv{\vec{v}}
\def \vz{\vec{z}}
\def\flp#1{{\left\langle#1\right\rangle}_p}
\def\T {\mathsf {T}}

\def\sfG {\mathsf {G}}
\def\sfK {\mathsf {K}}

\def\mand{\qquad\mbox{and}\qquad}

\title[Moments  of $L$-functions over subgroups]
{Moments and non-vanishing of $L$-functions over subgroups of optimal index}

\author[M. Munsch]{Marc Munsch}
\address{Department of Mathematics, University of  Jean Monnet,
23, rue du docteur Paul Michelon, 42 023 Saint-Etienne Cedex 2, France}

\email{marc.munsch@univ-st-etienne.fr}

\author[I. E. Shparlinski] {Igor E. Shparlinski}
\address{School of Mathematics and Statistics, University of New South Wales, Sydney, NSW 2052, Australia}
\email{igor.shparlinski@unsw.edu.au}


\begin{abstract}
We obtain an asymptotic formula for all moments of Dirichlet $L$-functions $L(1,\chi)$ modulo $p$ when averaged over a subgroup of characters $\chi$ of size $(p-1)/d$ with $\varphi(d)=o(\log p)$.  Assuming the infinitude of Mersenne primes, the range of our result is optimal and improves and generalises the previous result of S.~Louboutin and M.~Munsch (2022) for second moments. We also use our  ideas to get an asymptotic formula for the  second moment of $L(1/2,\chi)$ over subgroups of characters of similar size. This leads to non-vanishing results in this family where the proportion obtained depends on the height of the smallest rational number lying in the dual group.
This improves a recent result of this type due to 
{\'E}. Fouvry, E. Kowalski and Ph. Michel (2024).
 Additionally, we prove that, in both cases, we can take much smaller subgroups for almost all primes $p$. 
\end{abstract}  


\maketitle

\tableofcontents

\section{Introduction}

\subsection{Set-up and motivation}

A classical problem in analytic number theory is to understand the distribution of values of Dirichlet $ L$-functions $L(s,\chi)$ in the critical strip.  

On one hand, there are numerous works about the distribution of $L(1,\chi)$, see, for example~\cite{AMMP, BGGK,  CZ, DWZ, Elliott, GS-1, GS, Lam3, Lav, Lee19, MV} and references therein. On the other hand, the statistical distribution of $L(1/2,\chi)$ is less understood despite precise conjectures on the moments~\cite{CFKOS} and only the first 
moments~\cite{HB1,HB2,Sound2} are known unconditionally.
 
In this paper we are interested in moments of Dirichlet $L$-functions $L(s, \chi)$ over 
multiplicative characters $\chi$ lying in a non-trivial subgroup of characters.
This question for $s=1$ is motivated by the established connection with bounds on relative class numbers and averages of Dedekind sums in~\cite{Lou1,Lou2,LoMu1,LoMu2}. 
We also note that since we are mostly interested in individual bounds on $L$-functions, one may wish 
to take the set over which moments are computed to be as sparse as possible, see, 
for example,~\cite{KLM, KMN, LoMu1,LoMu2} for averaging over various subgroups of the character group. Let also mention that an average over restricted subsets of characters appeared in the recent breakthrough of 
Petrow and Young~\cite{PY1,PY2} establishing a Weyl-strength subconvex bound for all Dirichlet $L$-functions.
Here we continue this line of research and extend previously known results in two directions:
\begin{itemize}
\item we consider even moments of arbitrary order, which in turn allow us to study the distribution function of $L(1,\chi)$ 
with $\chi$ running over a subgroup of characters; 
\item  we consider averaging over subgroups of optimally small size, see~\eqref{eq:Opt Range}
(for both $s=1$ and $s=1/2$), 
see also Remarks~\ref{rem:M2 limit}  and~\ref{rem: L1/2 range}. 
\end{itemize}

More precisely, let  $ \F_p^*$  be 
the  multiplicative group of the finite field  $\F_p$  of $p$ elements and let $\cX_p$ be
the group of multiplicative characters of  $\F_p^*$.

Given an integer $m \ge 1$ with $m\mid p-1$  and  a multiplicative subgroup $\cG_m\subseteq \F_p^*$ of index $m$ and thus 
of order $(p-1)/m$, we denote by $\cX_{p,m}$ 
the group of multiplicative characters of  $\F_p^*$  which are trivial on $\cG_m$ and by $\cX_{p,m}^+$ the subgroup of $\cX_{p,m}$ of even characters,  that is,  
\[
\cX_{p,m}^{+} = \left\{\chi \in \cX_{p,m}:~\chi(-1)=1\right\}
\] \ and the set $\cX_{p,m}^{-}$    
of odd characters $\chi\in \cX_{p,m}$, that is,  
\[
\cX_{p,m}^{-} = \left\{\chi \in \cX_{p,m}:~\chi(-1)=-1\right\}.
\]

 We now define for $\nu \geq 1$ the associated moments 
\[
M_\nu(p,m)
=\frac{1}{m} \sum_{\chi \in \cX_{p,m} } \vert L(1,\chi)\vert^{\nu},  \quad  M_\nu^{-}(p,m)
=\frac{1}{\# \cX_{p,m}^{-}} \sum_{\chi \in \cX_{p,m}^{-}} \vert L(1,\chi)\vert^{\nu}.
\]

By~\cite[Chapter~4]{Wash}, for an odd $d = (p-1)/m$, there is a close connection between the second moment  $M_2^{-}(p, m) $ and the
relative class number of the imaginary subfield of the cyclotomic field $\Q\(\exp(2 \pi i/p)\)$ of  degree $m$ over $\Q$. Furthermore,  also for an  odd $d = (p-1)/m$,  these moments are closely 
related to correlation of {\it Dedekind sums\/}, and this connection underlines the approach 
of~\cite{Lou1,Lou2,LoMu1,LoMu2}. 

We also define the second moment of $L$-functions on the critical line, 
over subgroups of even characters:
\[
\fM_{2} (p,m)
=\frac{1}{\# \cX_{p,m}^+ } \sum_{\chi \in \cX_{p,m}^+} \vert L(1/2,\chi)\vert^{2}.
\]    
We would like to {\it emphasise\/}  that this restriction is imposed only to simplify the exposition and 
our results can be easily extended to odd characters.

We are especially interested 
in average values and non-vanishing over thin subgroups and some of our results are 
derived for subgroups of optimally large index. This is in contrast with recent work of 
Fouvry, Kowalski and Michel~\cite{FKM} whose method applies only to subgroups of bounded 
index. Hence, to the best of our knowledge, the moments  of  $L(1/2,\chi)$ 
and non-vanishing in thin subgroups have never been studied before. In particular, for almost all primes $p$ we obtain 
a positive proportion of non-vanishing in subgroups of size  $m \ge p^{5/6+\varepsilon}$ for any fixed $\varepsilon > 0$, 
while for groups of size $m = p^{1+o(1)}$, this proportion is $1/7+o(1)$ (in both cases $o(1)$ means a negative quantity, 
of course). 
We refer to Theorems~\ref{thm: nonvanishing}
and~\ref{thm:almostall 1/2} below for exact formulations of such non-vanishing results.

 The main novelty of our argument is contained in Lemma~4.3, which allow to bound efficiently  averages over a subgroup of optimal index of quantities depending on the height of the 
rational numbers inside the subgroup. This type of sum appears naturally while studying analytic problems in the family of Dirichlet characters lying in a subgroup and is likely to have further applications. As an example, it is used in a forthcoming work of Louboutin~\cite{Lou3} to derive an asymptotic formula for the logarithms of the
relative class numbers of the imaginary abelian number fields of large degrees. 
We briefly outline the techniques involved in the proof of Lemma~4.3  in Section~\ref{sec:res & methods}.

Our argument can also be easily adapted to produce full analogues of our  present results
for cosets $\chi^* \cX_{p,m}$ 
with a fixed  non-principal character $\chi^* \in \cX_p$.


\subsection{Summary of  our results and methods}  
\label{sec:res & methods}

On one hand we give asymptotic formulas for all the  moments  $M_{2k}(p,m)$ and $M_{2k}^-(p,m)$ for all primes $p$ and  all $m = (p-1)/d$ for a divisor $d \mid p-1$, which is not too large, roughly, of order less than $\log p$, see Theorem~\ref{theo-kth}
for an exact statement and range. 
In fact, assuming the infinitude of Mersenne primes, our range is 
optimal, see Remark~\ref{rem:M2 limit}  below. Moreover, we also show that, for almost all primes, we can compute all moments for substantially smaller subgroups, that is,  as soon as $m \ge p^{2/3+\varepsilon}$ with a fixed $\varepsilon>0$, see Corollary~\ref{cor:Conj Almost All}. The results on $M_{\nu}^{-}(p, m) $ can be used to shed more light on the behaviour of averages of Dedekind sums, see~\cite{BMS}.

On the other hand, it is now well-known that a remarkable array of results can be obtained as soon as one has an asymptotical formula with power saving for the second moment of the values of $L$-functions on the critical line (see~\cite{BKMSF} for a general account and applications of this philosophy).

We obtain in Theorem~\ref{thm:secondmoment} an asymptotic formula for the second moment $\fM_2 (p,m)$ (more generally, for twisted moments, see Lemma~\ref{asymp-Bm}).  
 Instead of the classical power-saving, the error terms involved depends on the size of small solutions to certain linear congruences. As an application, we use the \textit{mollifiers} introduced by Iwaniec and Sarnak~\cite{IS} to obtain in Theorem~\ref{thm: nonvanishing} a lower bound on the frequency of non-vanishing of $L(1/2, \chi)$ over a `thin''  subgroup 
of characters where the proportion of non-vanishing depends on the size of solutions of these congruences. Moreover, for almost all primes $p$, we show in Theorem~\ref{thm:almostall 1/2} that  
there is a positive proportion of non-vanishing for rather small subgroups. 

 As mentioned above, our analysis of $M_{2k}(p,m)$ and $\fM_2 (p,m)$ requires the study of small solutions to some families of linear congruences. More precisely we study the distribution of fractions lying in a subgroup of small order.
In turn, this  has led us to  refining a  result about product sets of Farey fractions from~\cite{BKS}, which has a flavour coming from {\it additive combinatorics\/}.

\section{Main results}

\subsection{Setting}
In this paper, for  the sake of simplicity, we always assume that 
\[
d=(p-1)/m
\]
and that $d$ is {\it odd\/}.  Hence $\cX_m$ always contains $m/2$ odd and $m/2$ even characters, which makes the normalisation more straightforward. In particular, 
\[
 M_\nu^{-}(p,m)
=\frac{2}{m} \sum_{\chi \in \cX_{p,m}^{-}} \vert L(1,\chi)\vert^{\nu},  
\quad  \fM_{2} (p,m) =\frac{2}{m} \sum_{\chi \in \cX_{p,m}^+} \vert L(1/2,\chi)\vert^{2} .
\]  

We would like to emphasise that our methods apply, after just some straightforward typographical changes, to even values of $d$ as well.

\subsection{Moments and distribution of $L(1, \chi)$ over small subgroups}

Let $\tau_k$ denote the $k$-fold divisor function, that is 
\[
\tau_k(n)=\sum_{\substack{n_1\cdots n_k=n \\ n_i \geq 1}} 1.
\] 
For an integer $k\geq 1$ we also define the constant
\[a(k)=\sum_{n=1}^{\infty} \frac{\tau_{k}^2(n)}{n^2}.\]  
It is well known that $M_{2k}(p,p-1) =  a(k)   + o(1)$; for example, it 
follows from~\cite[Theorem~1.1]{Lee19} (taken with $a=b=k$, $c_1 = c_2 = 1$).

\begin{thm}\label{theo-kth}
 Let $k$ be a fixed integer and let $\kappa < 1$ be an arbitrary constant. 
 For any  $m \mid p-1$, we have 
\[
M_{2k}(p,m) , M_{2k}^-(p,m) =  a(k)  +   O\(p^{-\kappa/\varphi(d)} + p^{-1/4+o(1)} \). 
\]
 \end{thm} 

 \begin{rem}
\label{rem:M2 limit}  We see that Theorem~\ref{theo-kth} is nontrivial provided that  
\begin{equation}\label{eq:Opt Range}  
\varphi(d)= o(\log p),
\end{equation}    
in which case  we have 
\begin{equation}\label{eq:M2pm}  
M_{2k}(p,m), M_{2k}^-(p,m) =  a(k)  + o(1), \qquad \text{as}\ p\to \infty.
\end{equation}  
In fact 
the range~\eqref{eq:Opt Range}  
is the best we can hope even for $k=1$. Indeed,  
by~\cite[Theorem~5.4]{LoMu2} and assuming that there are infinitely many Mersenne primes $p=2^d-1$,  the asymptotic 
formula~\eqref{eq:M2pm} does not hold in the range $\varphi(d)= d-1 \asymp \log p$.  
\end{rem}

 \begin{rem}
\label{rem:Gen}  It is easy to see from the proof  that  both error terms in Theorem~\ref{theo-kth} are
 $O\(\vartheta(m,p)^{-\kappa} + d p^{-1/4+o(1)} \)$,  where $\vartheta(m,p)$ is defined below in~\eqref{eq:def-theta}.
\end{rem}  

 \begin{rem}
 A  result, similar to that of Theorem~\ref{theo-kth} with a restricted range for $d$ has been announced in~\cite[Section~9, Theorem~B]{Lam} as proved in an unpublished manuscript of A.~Granville and K.~Soundararajan.
\end{rem}  

 Let 
\[ \psi_p(x)=\frac{1}{p-1} \# \left\{ \chi \in \cX_p^*:~\vert L(1,\chi)\vert \leq x \right\}, 
\]  
where,  $\cX_p^* $ is the set of all non-principal characters modulo $p$.

 A classical application of the method of moments (see, for example,~\cite[Theorem~5.2]{Barban}) implies that 
\[
 \lim_{p \to \infty} \psi_p(x) = \psi(x),
\]
where $\psi(x)$ is a distribution function with characteristic function of the form 
\[
 f(x)=\sum_{k=0}^{\infty}\frac{a(k)}{k!}(ix)^k, 
\]  
where $a(k)$ are  as in Theorem~\ref{theo-kth}.  
Given an integer $m$ with $m\mid p-1$  we define similarly
\[
 \psi_{p,m}(x)=\frac{1}{m}\# \left\{ \chi \in \cX_{p,m}:~\vert L(1,\chi)\vert \leq x \right\}.
\]
It follows from Theorem~\ref{theo-kth}, that for any fixed $x \in \R$,  in the range $d=o(\log p)$, we have 
\[ \lim_{p \rightarrow \infty} \psi_{p,m}(x) = \psi(x).
\]
In other words, the distribution of $\vert L(1,\chi)\vert$, while $\chi$ runs over odd characters in a sufficiently large subgroup of $\cX_p$, has the same limiting behaviour as $\vert L(1,\chi)\vert$ while $\chi$ runs over $\cX_p$. Let us add that Granville and Soundararajan~\cite{GS,GS2} 
have proved that the tails of the distribution function $\psi$ have double-exponential decay in a wide range of values of $x$, allowing them to prove interesting results about extreme values of $\vert L(1,\chi)\vert$.

Next, we show that for almost all primes one can get a much better bound on the error term in the asymptotic formula
of Theorem~\ref{theo-kth}. In fact, although we believe that the range of $d = (p-1)/m$ in Theorem~\ref{theo-kth} is
optimal, see Remark~\ref{rem:M2 limit}, for almost all $p$ it can be significantly expanded. 

\begin{thm}\label{thm:almostall}
Let $\kappa < 1$ and  let an integer $k\ge 1$ be fixed. 
For arbitrary  $D,Q, R \ge 2$, for all except  at most $O\(D^2R (\log R)^2 \log Q\)$ primes  $p\in [Q,2Q]$, 
for any  $m \mid p-1$   such that  $3 \le d \le D $, we have 
\[
M_{2k}(p,m) , M_{2k}^-(p,m) =  a(k)  +   O\(D^{1-\kappa} R^{-\kappa} + Dp^{-1/2+o(1)}\). 
\] 
 \end{thm}

In particular,  we see that for almost all primes in dyadic intervals, Theorem~\ref{theo-kth} holds for much smaller subgroups.

 \begin{cor}\label{cor:Conj Almost All} 
 Let $0<\varepsilon<1$ and let an integer $k\ge 1$ be fixed.  
For a  sufficiently large $Q$, for any $d$  such that 
$3\le d \le p^{1/2-\varepsilon} $, we have,
\[
M_{2k}(p,m) , M_{2k}^-(p,m) =  a(k)  +   o(1), \qquad \text{as}\ Q\to \infty,
\]
for all except  at most $Q^{1-\varepsilon}$ primes $p\in [Q,2Q]$.
 \end{cor}

\subsection{Second moment and nonvanishing of $L(1/2, \chi)$ over small subgroups}
\label{sec:L1/2}

 We study the classical nonvanishing problem of Dirichlet $L$-functions along  ``thin'' subgroups of characters. It is conjectured that $L(1/2, \chi) \neq 0$ for every primitive Dirichlet character $\chi$. Using mollifiers, Balasubramanian and Murty~\cite{BM} have proved that there is a small positive proportion of characters $\chi \bmod p$ such that $L(1/2, \chi) \neq 0$. This method has proven to be successful to obtain a positive proportion of nonvanishing of $L$-functions in many contexts, see~\cite{KMV, IS2, Sound1} as well as the works discussed below. 
 
The mollifier method requires the understanding of the first and second twisted moments. We study twisted moments of $L(1/2, \chi)$ over a subgroup of $\cX_p$ and, as a byproduct of our results,  obtain an asymptotic formula for the second moment which generalises the result of Heath-Brown~\cite{HB1} obtained for the full group $\cX_p$.  

We remark that similar averages over subgroups  $ \cX_{p,m}$ of fixed index $d = (p-1)/m$ 
have recently been studied by Fouvry, Kowalski and Michel~\cite{FKM}. Here we obtain stronger and 
more uniform versions of some results from~\cite{FKM} which hold for subgroups of optimally large index, 
see Remark~\ref{rem: L1/2 range} below.

As usual, we use $\Gamma(x)$ to denote the  standard $\Gamma$-function  and let 
\[
\gamma = 0.57721\ldots
\]
denote the  {\it Euler--Mascheroni constant\/}. The following asymptotic formula is 
a special case of Lemma~\ref{asymp-Bm}. 

\begin{thm}\label{thm:secondmoment}
Let $\kappa <1/2$. For any $m \mid p-1$, we have
\begin{align*}
\frac{2}{m}\sum_{\chi \in \cX_{p,m}^+} \vert L(1/2,\chi)\vert^2 = \log \left(\frac{p}{\pi}\right) &  +2\gamma + \frac{\Gamma'}{\Gamma}\left(\frac{1}{4}\right) \\
& \quad + O\(p^{-\kappa/\varphi(d)} \log p+ p^{-1/8 + o(1)} \).
\end{align*}
\end{thm}

\begin{rem}  
\label{rem: L1/2 range}
Under the condition~\eqref{eq:Opt Range}, 
by Theorem~\ref{thm:secondmoment} we have 
\[
\frac{2}{m}\sum_{\chi \in \cX_{p,m}^+} \vert L(1/2,\chi)\vert^2 \sim \log p,
\]
and a more precise relation holds in the range $\varphi(d) = o\(\log p/\log \log p\)$.  
\end{rem}  


We are now ready to present  our main result, which  
shows that the proportion of characters in a subgroup $\cX_{p,m}$ such that $L(1/2,\chi) \neq 0$ depends explicitly on the size of solutions of certain linear congruences. 
We remark that it does not seem to have any predecessors besides a less uniform result of Fouvry, Kowalski and Michel~\cite{FKM}. The main  novelty of our approach is that under some natural conditions (and also for almost all moduli) it is able to produce a positive proportion of non-vanishing values
for rather thin subgroups of characters, 
see Remark~\ref{rem:nonvanishing} and Corollary~\ref{cor:Conj Almost All 1/2}.  A result of this type, that is, for subgroups of growing index, 
cannot be achieved within the approach of~\cite{FKM}.

To formulate our result, given $\lambda\in \Z$  and a prime $p\nmid \lambda$, we define
\begin{equation}\label{eq:def-rho} 
\rho(\lambda,p)
=\min  \{rs:~(r,s) \in  {\mathbb N}^2 \setminus \{(0,0)\}, \ r\equiv   \lambda s \bmod p\}.
\end{equation}It is also convenient to define 
\begin{equation}\label{eq:def-theta} 
\vartheta(m,p)= \min_{\substack{\lambda \in \cG_m\\\lambda \ne 1}} \rho(\lambda,p). 
\end{equation}

\begin{thm}\label{thm: nonvanishing}
 Let  $\varepsilon>0$ be fixed. 
For  any $d$ such that  $1 \le d  \le  p^{1/8-\varepsilon} $, we have  
\[
\frac{1}{m} \sum_{\substack{\chi \in \cX_{p,m} \\L(1/2,\chi) \neq 0}}  1 \ge c(\varepsilon) \frac{\log\vartheta(m,p)}{\log p}, 
\]
where $c(\varepsilon) >0$ depends only on $\varepsilon$. 
\end{thm} 

\begin{rem}  
\label{rem:nonvanishing} 
 It follows from Lemma~\ref{lem:rho2 - extreme} below that there is a proportion of order at least  $1/\varphi(d)$ of characters  in the subgroup $\cX_{p,m}$ such that $L(1/2,\chi) \neq 0$. 
Assuming~\eqref{eq:Opt Range},   this improves upon the proportion 
of order $(\log p)^{-1}$ obtained by studying the first two moments of $L(1/2,\chi)$ inside this family. Moreover, assuming that there exists $\delta>0$ such that $\vartheta(m,p)\gg p^{\delta}$, we obtain a positive proportion of non-vanishing.  \end{rem}

For almost all primes, we can increase the length of the mollifier  used in our method and prove the following result.

\begin{thm}\label{thm:almostall 1/2}
Let $(\alpha,\beta,\eta)$ be such that 
\[
0<\alpha<1/6, \qquad \beta <1/6-\alpha, \qquad  \alpha+ 4\beta < \eta.  
\] 
Then for all except at most $Q^{2\alpha + \eta+ o(1)}$ primes  $p\in [Q,2Q]$, for any 
$d$ such that  $1 \le d \le p^{\alpha}$, we have 
\[
\frac{1}{m} \sum_{\substack{\chi \in \cX_{p,m} \\L(1/2,\chi) \neq 0}}  1 \geq \frac{\beta}{1+\beta}.
\]
 \end{thm}

Choosing $\alpha= 1/6-\varepsilon$, $\beta=\varepsilon/2$, $\eta=1/6+\varepsilon$ leads to the following result.

 \begin{cor}\label{cor:Conj Almost All 1/2} Let $Q\ge 2$ be sufficiently large and let $0<\varepsilon<1$. Then,  
 for all except  at most $Q^{1/2}$ primes $p\in [Q,2Q]$,   
 for any $d$ such that  
 $1 \le d \le  p^{1/6-\varepsilon}$, we have 
\[\frac{1}{m} \sum_{\substack{\chi \in \cX_{p,m} \\L(1/2,\chi) \neq 0}}  1  \ge c(\varepsilon), 
\]
where $c(\varepsilon) >0$ depends only on $\varepsilon$. 
 \end{cor}

\section{Preliminaries} 

\subsection{Notation and conventions}
\label{sec:not}

We adopt the Vinogradov notation $\ll$,  that is,
\[
A\ll B~\Longleftrightarrow~ B\gg A~\Longleftrightarrow~A=O(B)~\Longleftrightarrow~|A|\le cB, 
\]
for some constant $c>0$. This constant $c$ may sometimes, depend on the real parameter $\varepsilon>0$ and occasionally,  where obvious, on the   integer parameter $k\ge 1$.

For a real $A> 0$ we write $a \sim A$ to denote  that $A  < a \le 2A$. 

For a finite set $\cS$ we use $\# \cS$ to denote its cardinality.

As usual, we use $\tau(r)$ and $\varphi(r)$ to denote the divisor and 
the Euler function, of an integer $r \ge 1$, espectively.  It is useful to recall the well-known estimates
\[
\tau(r) = r^{o(1)}\mand  \varphi(r) \gg \frac{r}{\log \log r}
\]
as $r\to \infty$, which we use throughout the paper, 
see~\cite[Theorems~5.4 and~5.6]{Ten} for more precise estimates.

\subsection{Average values of the divisor function in short arithmetic progressions}
We need the  special case of a much more general result of Shiu~\cite[Theorem~2]{Shiu}, 
which we take with 
\[
\alpha = \beta = 1/4, \qquad \lambda = 1, \qquad y = x.
\]
Using the trivial bound $\varphi(m)/m \le 1$, for any integer $m \ge 1$, we  derive
the following estimate.

\begin{lemma}
\label{lem:Shiu}
For any fixed integer $k$, uniformly over integers $m\ge 1$ and $(a,m)=1$ and a real $t\ge m^2$, we have  
\[
\sum_{\substack{s\le t\\ s \equiv a \bmod m}} \tau_k(s)\ll  \frac{t}{m}\( \log t\)^{k-1}.
\]
\end{lemma}

\subsection{Product sets of Farey fractions}
For a positive integer $Q$, let \[ \mathcal{F}_Q = \left\{ \frac{r}{s}:~r,s \in  \mathbb{N}, \ \gcd(r,s)=1,\ 1\le r,s \leq Q\right\} \] be the set of Farey fractions of order $Q$ (we note that we discarded the usual restriction $r \le s$).

The following result has been given in~\cite[Corollary~3]{BKS}: let $\mathcal{A} \subseteq \mathcal{F}_Q$ and $k\in \mathbb{N}$ be a given integer, then the $k$-fold product set 
\[
\mathcal{A}^{(k)} = \{a_1\cdots a_k:~ a_1,\ldots, a_k \in \cA\} 
\] 
of $\mathcal{A}$  satisfies
 \[   \# (\mathcal{A}^{(k)}) > \exp\left(-C(k) \frac{\log Q}{\sqrt{\log  \log Q}} \right)(\# \mathcal{A})^k .\]
It has also been noticed in~\cite[Section~2.2]{BKS} that one  can take $C(k) \leq Ck \log k$ for $k\geq 2$ and a sufficiently large
absolute constant $C$. 

We now show that incorporating the improvement of~\cite[Lemma~2]{BKS} given by Shteinikov~\cite[Theorem~2]{Sht} in the case $k=2$ and then using double  induction we can in fact prove a result with better dependence on $k$,  which can be on independent interest  and is  also   important for our argument.  

\begin{lem}\label{lem:kfoldproduct} 
There is an absolute constant $C> 0$ such that for a sufficiently large $Q$, 
uniformly over integers $k \ge 1$ we have 
\[
  \# (\mathcal{A}^{(k)}) > \exp\left(-Ck \frac{\log Q}{\log  \log Q} \right)(\# \mathcal{A})^k .
  \]
  \end{lem}

  \begin{proof} For $k=2$ the result is given in~\cite[Theorem~2]{Sht}, which asserts that for any sets  $\cA, \cB \subseteq \mathcal{F}_Q$, for their product set we have 
\[
\# \(\cA \cB\) \ge   \exp\(-C \frac{\log Q}{\log  \log Q}\) \# \cA \# \cB, 
\]
for  some constant $C>1$.

Next we note that if $k = 2^i$, then  using that $C+ 2 C(2^{i-1}-1) = C(2^{i}-1)$, it is easy to show that 
\[
  \# (\mathcal{A}^{(2^i)}) >  \exp\(-C(2^{i}-1) \frac{\log Q}{\log  \log Q}\)\(\# \mathcal{A}\)^{2^{i}}  .
\]
 Indeed, this is true for $i=1$. Then for $i \ge 2$ by induction
\begin{align*}
  \# \(\mathcal{A}^{(2^i)}\) &  =  \# \(\mathcal{A}^{(2^{i-1})} \mathcal{A}^{(2^{i-1})}\)  \\
  & >  \exp\(-C \frac{\log Q}{\log  \log Q}\) \\
  & \qquad \qquad \( \exp\(-C \(2^{i-1}-1\) \frac{\log Q}{\log  \log Q} \)(\# \mathcal{A})^{2^{i-1}}\)^2\\
  &=  \exp\(-C(2^{i}-1) \frac{\log Q}{\log  \log Q}\)\(\# \mathcal{A}\)^{2^{i}}  .
\end{align*}
  We now write $k$ in binary as 
  \[
  k = \sum_{j=1}^s 2^{i_s}, \qquad i_s > \ldots > i_1 \ge 0.
  \]
  Then it is easy to see (again by induction but this time on $s$) that  we have 
  \begin{align*}
  \# (\mathcal{A}^{(k)}) &>  \exp\(-C \frac{\log Q}{\log  \log Q}\)  
    \# \(\mathcal{A}^{(k-2^{i_s})} \)   \# \(\mathcal{A}^{(2^{i_s})}\)\\
& \ge  \exp\(-\(C + C  \(k-2^{i_s}\) + C(2^{i_s}-1)\)\frac{\log Q}{\log  \log Q} \)\(\# \mathcal{A}\)^{k} \\
  &\ge   \exp\(-C k \frac{\log Q}{\log  \log Q} \) \(\# \mathcal{A}\)^{k}  , 
\end{align*}
which concludes the proof. 
  \end{proof}

\begin{rem} By~\cite[Theorem~2]{Sht}, one can take any $C > 8 \log 2$ provided that $Q$ is sufficiently large, 
however the result is false with $C <  4 \log 2$. \end{rem}

\section{Small solutions to linear congruences}

\subsection{Bounds for all primes}

We recall the definition~\eqref{eq:def-rho}. We note that a higher-dimensional analogue of $\rho(\lambda,p)$
(which allows negative integers) is well-studied in the context of  the theory of pseudo-random generators and optimal coefficients; we refer for more information to the
work of Korobov~\cite{Kor1,Kor2}   
and  Niederreiter~\cite{Nied}.

Here we concentrate on the two-dimensional case.  First we recall that by~\cite[Lemma~4.6]{LoMu1} we have the following lower 
bound on $\rho(\lambda,p)$.

\begin{lemma}
\label{lem:rho2 - extreme} 
 Let $\cG_m\subseteq \F_p^*$ be  of index $m$ and order $d =(p-1)/m$, then we have
\[
\vartheta(m,p) \gg p^{1/\varphi(d)}.
\]
\end{lemma}  

\begin{proof} If $\lambda=-1$, we have the trivial lower bound $\rho(-1,p) \gg p$. 

Now let $\lambda \in \F_p^*$ be of order  $e\ge 3$. Then  by~\cite[Lemma~4.6]{LoMu1}  we have 
\[
\rho(\lambda,p) \gg p^{1/\varphi(e)}.
\]
Since for any  $\lambda\in  \cG_m \setminus \{1, -1\}$ the order $e$ satisfies $3 \le e \le d$, the 
result follows. 
\end{proof} 

Next we  estimate the number of values of  $\lambda\in  \cG_m$ with $\rho(\lambda,p) \le \ell$ 
for some parameter $\ell \ge 1$.

\begin{lem}\label{lem:fractions} 
 For a multiplicative subgroup $\cG_m\subseteq \F_p^*$ 
of index $m$ and order $d =(p-1)/m$, for any integers $k \ge 1$ and $\ell\ge 3$ 
such that $\ell^{2k}< p/2$, we have 
\[
\# \left\{\lambda \in \cG_m:~\rho(\lambda,p) \le \ell   \right\} \leq  
d^{1/k}\exp\left(C \frac{\log \ell}{\log \log \ell}\right), 
\]
where $C$ is an absolute constant.
\end{lem} 

\begin{proof}
As in~\cite{BKS}, we consider the set 
\[ \mathcal{V}(p, \cG_m, \ell) =  \left\{\lambda\in  \cG_m:~\lambda \equiv r/s \bmod p  \text{ with } |r|, |s| \leq \ell \right\}. \]
Hence 
\[
\# \left\{\lambda \in  \cG_m:~ \rho(\lambda,p) \le \ell   \right\} \leq \# \mathcal{V}(p, \cG_m, \ell).
\]  
Incorporating Lemma~\ref{lem:kfoldproduct} in the proof of~\cite[Lemma~4]{BKS} we obtain  
that for any positive integers $k$ and  $\ell$ such that $\ell^{2k}< p/2$  the following bound holds:
 \[ \mathcal{V}(p,\cG_m, \ell)  \leq  d^{1/k}\exp\left(C \frac{\log \ell}{\log \log \ell}\right)\]
 and the result follows. 
\end{proof}

Most of our results rely on bounds for the sums 
\begin{equation}\label{eq:  Rmp} 
\sR_\alpha (m,p) =  \sum_{\substack{\lambda \in \cG_m\\\lambda \ne 1}} \rho(\lambda,p)^{-\alpha} ,  \qquad \qquad \alpha>0. 
\end{equation}

We  now recall the definition~\eqref{eq:def-theta}.

\begin{lem}\label{lem:Sum rho}  
Let $\alpha, \beta> 0$ be fixed with $4\beta < \alpha$.
  For a multiplicative subgroup $\cG_m\subseteq \F_p^*$ 
of index $m$ and order $d =(p-1)/m \le  p^\beta $, for any fixed $0<\kappa < \alpha- 4 \beta $
  we have 
\[
\sR_\alpha (m,p) \ll  \vartheta(m,p)^{-\kappa} + d p^{-\alpha /4}.
\]
\end{lem}

\begin{proof}  
Clearly for any 
$\lambda \in \cG_m\setminus \{1\}$ the value $\rho(\lambda,p)$ belongs to at least one of the possibly overlapping 
intervals 
\[
\cI_n=\left[ e^n,  e^{(n+1)}\right], \qquad n =L, \ldots,  U,
\]
where, using that $\rho(\lambda,p) \le \min\{p-\lambda, \lambda\} < p/2$, we can set 
\[
L = \fl{ \log \vartheta(m,p)} \mand  U = \rf {\log (p/2)}. 
\] Moreover it follows from Lemma~\ref{lem:fractions} that 
for any $\nu \ge 1$ such that $e^{2(n+1)\nu } < p/2$, we have 
\begin{equation}\label{eq: count rho in In} 
\# \left\{\lambda \in \cG_m:~\rho(\lambda,p)  \in \cI_n   \right\}  \leq  d^{1/\nu}\exp\(C \frac{ n+1}{\log  \(n+1\)}\). 
\end{equation}
Now, for $n  \le U/4-1$ we choose  
\[ \nu= \fl{\frac{\log (p/2)}{2(n+1)}}.\] 
One easily verifies that for $n  \le U/4-1$ we have 
\[
\frac{\log (p/2)}{2(n+1)} \ge 2
\]
and thus  $\nu \ge 2$, which implies 
\[
\nu   \ge   \frac{\log (p/2)}{4(n+1)}.
\]

We now  see   from~\eqref{eq: count rho in In}  that for $n \leq   U/4-1 $
\begin{equation}
\begin{split}
\label{eq: small rho} 
 & \# \left\{\lambda \in \cG_m:~\rho(\lambda,p)  \in \cI_n   \right\} \\
 & \qquad \qquad \qquad   \ll   
\exp\(\frac{4\log d}{\log (p/2)}  (n+1)  + C \frac{ n+1}{\log  \(n+1\)}\)   \\
& \qquad \qquad \qquad  \ll  \exp\(\frac{4\log d}{\log p} n  + 2C \frac{n}{\log  n}\).
  \end{split}
\end{equation}

 Next, for $n > U/4 -1$, we use the trivial bounds 
\[ 
\sum_{ U/4-1 \le  n\le U} \# \left\{\lambda \in \cG_m:~\rho(\lambda,p)  \in \cI_n   \right\} \leq d
\]
and 
$
\rho(\lambda,p) \ge e^{ U/4} \gg p^{1/4}.
$
Hence the total contribution to $\sR_\alpha (m,p)$ from such values of $n$ can be bounded as
\begin{equation}\label{eq: large rho} 
\sum_{U/4-1\le  n\le U} \# \left\{\lambda \in \cG_m:~\rho(\lambda,p)  \in \cI_n   \right\} \rho(\lambda,p) ^{-\alpha } 
\ll d p^{-\alpha /4}.
\end{equation}

We now average over $\lambda \in \cG_m$ and combining~\eqref{eq: small rho}  and~\eqref{eq: large rho} 
we derive 
\[
\sR_\alpha  (m,p)  \ll  \sum_{L\le n \le U/4 -1 } \exp\(\frac{4\log d}{\log p}  n  + 2C \frac{n}{\log  n}- n\alpha \) 
+ \frac{d}{p^{\alpha /4}}.
\]
By our assumption, 
\[
\frac{4\log d}{\log p}   \le 4  \beta < \alpha
\] 
and since  we can assume that   $L \to \infty$ (as otherwise the bound is trivial), we obtain
\begin{align*}
\sR_\alpha (m,p)  & \ll    \sum_{n \geq L}  \exp\left(-\kappa  n\right)  +   d p^{-\alpha /4} \\
&  \ll \exp(-\kappa L)  +  d p^{-\alpha/2}  \ll    \vartheta(m,p)^{-\kappa} +  d p^{-\alpha /4},
\end{align*} 
for any $\kappa$ with 
$
0<\kappa < \alpha-   4 \beta,
$
which  concludes the proof.
\end{proof}

Finally, we recall the following upper bound, given by Cilleruelo and  Garaev~\cite[Lemma~2]{CillGar}, 
on the number of small primitive solutions to linear congruences. 

\begin{lemma}\label{lem: ModHyperb} 
For any   real  $z_1, z_2 \ge 1 $ and a prime $p$ and $\lambda \in \F_p$, 
\[
\#\{(r,s):~ r \sim z_1, \ s \sim z_2, \ \gcd(r,s)=1, \ r \equiv \lambda s \bmod p\}  \ll  1 + \frac{z_1z_2}{p}. 
\]
\end{lemma}

\subsection{Bounds for almost  all primes}

For real numbers  $D, L, R \ge 2$ we define $\cE(D,L, R)$ as the set of primes $p$, for which  for some   $d < D$ 
there exist at least $L$ elements $\lambda \in \F_p^*$  of  order $d \geq 3$ and such that 
\begin{equation}\label{eq: rho<R} 
\rho(\lambda,p) \le  R. 
\end{equation}

\begin{lemma}\label{lem: almost all} 
For any   real  numbers $D,  L, R \ge 2$, we have 
\[
\# \cE(D,L, R) \ll  D^2L^{-1}  R (\log R)^2. 
\]
\end{lemma}

\begin{proof}  We start with following the proof of~\cite[Lemma~4.6]{LoMu1}.
Let, as usual,  
\[\varphi_d(X)
=\sum_{ j=0}^{\varphi(d)} a_jX^j 
=\prod_{\substack{1\leq j\leq d\\\gcd (d,j)=1}}\(X- \exp\(2 \pi ij/d\)\)
\]
denote the $d$-th cyclotomic polynomial. 

We clearly have 
\begin{equation}
\label{eq:Phi_d = 0}
\varphi_d(\lambda)\equiv 0 \bmod p.
\end{equation}

We fix some $3 \le d < D$ and let $p$ be a prime  for which  there 
exist at least $L$ elements $\lambda \in \F_p^*$  of  order $d \geq 3$ with~\eqref{eq: rho<R}.

Let $\rho(\lambda,p)  = |h_1| |h_2|  \le  R$ for some integers $h_1, h_2$.  
Clearly  we have $\gcd(h_1, h_2)=1$. 

We now define a polynomial 
$$
F_R(X) = \prod_{\substack{|h_1| |h_2| \le R\\\gcd(h_1, h_2)=1}} \(h_1X - h_2\).
$$
Clearly $F_R(X)$ is square-free of degree $\deg F_R \ll R \log R$, since there are at most $O(R\log R)$ pairs of non-zero integers $(h_1, h_2)$ with $|h_1| |h_2| \le R$
by the known bound on the average value of the divisor function (see, for 
example~\cite[Equation~(1.80)]{IwKow}.

We see from~\eqref{eq:Phi_d = 0} that $\varphi_d(X)$ and $F_R(X)$ have at least $L$
common roots modulo $p$. Thus, by~\cite[Lemma~5.3]{KoSh} we have 
\begin{equation}
\label{eq:p-div}
p^{\rf{L}} \mid \sR
\end{equation}
where $\sR = \Res\(F_R(X),\varphi_d(X)\)$ is the  the resultant of polynomials $F_R(X)$ and $\varphi_d(X)$. 
Using that $\sR \ne 0$ (since for $d \ge 3$, $\varphi_d$ has no rational roots) and that 
 \begin{align*} 
|\sR| & =  \prod_{\substack{1\leq j\leq d\\\gcd (d,j)=1}} | F_R\(\exp\(2 \pi ij/d\)\)| \\
& \le(2R)^{\varphi(d) \deg F_R}  \le \exp\(O\(d R (\log R)^2\)\). 
\end{align*} 

Since the  product of any $s$ distinct primes is of size at 
least $2^s$, we deduce from the above bound on $|\sR| $ that the number $E_d$ of primes $p$ satisfying~\eqref{eq:p-div}
can be bounded as 
$$
E_d \ll dL^{-1}  R (\log R)^2 . 
$$
Summing over all $d \le D$ we conclude the proof. 
\end{proof}


 \section{Proofs of results on moments of $L(1,\chi)$}
\subsection{Preliminary comments} 
As before, let   $\cG_m\subseteq \F_p^*$ be  a multiplicative subgroup of index $m$ and let $\xi_m$ be the characteristic function of $\cG_m$.  We assume that $d$ is odd, $m$ is even and thus $-1 \notin \cG_m$ as otherwise there is no odd characters in $\cX_{p,m}$. 

Then,  if we restrict attention to characters of a given sign $a=0,1$  we have for $\gcd(rs,p)=1$ the following version of the identity~\cite[Equation~(2.1)]{Sound2} which follows from the 
orthogonality of characters: 
\begin{equation}\label{eq: even m}
 \sum_{\substack{\chi \in \cX_{p,m}\\\chi(-1) = (-1)^a}} \chi(r) \overline\chi(s) =
  \frac{m}{2} \( \xi_m(r/s) +   (-1)^a  \xi_m(-r/s) \). 
\end{equation}

We recall that the implied constants  in this section may depend on the integer $k$.

\subsection{Proof of Theorem~\ref{theo-kth}}

Assume that $\chi$ is not a quadratic character. Then, it follows from~\cite[Lemma~2.3]{GS} that
\[ L(1,\chi)^k=\sum_{r \geq 1}\chi(r)\frac{ \tau_k(r)}{r} e^{-n/Z} + O\left( \frac{1}{p}\right) \] with 
\[
Z= \exp((\log p)^{10}).
\]   
 Thus,
\[ \vert L(1,\chi) \vert^{2k} = \sum_{r,s\geq 1}\chi(r)\overline{\chi(s)
}\frac{ \tau_k(r)\tau_k(s)}{rs} e^{-(r+s)/Z} + O\left( \frac{(\log p)^{10k}}{p}\right), \] 
where we have used the bound
\begin{equation}\label{bound-tauk}\sum_{r\geq 1}\frac{ \tau_k(r)}{r} e^{-r/Z} \ll (\log 3Z)^k,\end{equation}
see, for instance,~\cite[Equation~(2.4)]{GS}.

Using~\eqref{bound-tauk}, the well-known bound $\vert L(1,\chi)\vert \ll \log p$ for the unique quadratic character modulo $p$
(see, for example,~\cite{GS-1}) and  the orthogonality relations~\eqref{eq: even m} we derive that 
 \begin{align*} 
&\sum_{\chi\in  \cX_{p,m}^-}\vert L(1,\chi)\vert^{2k} \\
& \quad =\#   \cX_{p,m}^-\sum_{\lambda \in \cG_m} (W^+_\lambda - W^-_\lambda) + O\left( \frac{(\log p)^{10k}}{d} + (\log p)^{2k}+(\log Z)^{2k}\right),
\end{align*} 
where 
\[
W^\pm_\lambda = \sum_{\substack{r,s\geq 1\\ r\equiv  \pm \lambda s \bmod p}}
\frac{ \tau_k(r)\tau_k(s)}{rs} e^{-(r+s)/Z}. 
\]  
Thus, we obtain
\begin{equation}\label{eq:Mk W} 
 \begin{split} 
\sum_{\chi\in   \cX_{p,m}^-}&\vert L(1,\chi)\vert^{2k} \\
& \quad = \#   \cX_{p,m}^- \sum_{\lambda \in \cG_m} (W^+_\lambda - W^-_\lambda) + O\left( (\log p)^{20k}\right).
\end{split} 
\end{equation} 
In the case of $M_{2k}(p,m)$, an identical argument also gives
\begin{equation}\label{eq:Mk W-} 
 \begin{split} 
\sum_{\chi\in  \cX_{p,m}}&\vert L(1,\chi)\vert^{2k} \\
& \quad = \#  \cX_{p,m} \sum_{\lambda \in \cG_m} W^+_\lambda  + O\left( (\log p)^{20k}\right).
\end{split} 
\end{equation} 
The error term in~\eqref{eq:Mk W} and~\eqref{eq:Mk W-}  is acceptable. Hence we now focus on the  sums $W^\pm_\lambda$. We split the discussion into two cases depending on whether $\lambda=1$ or not.

For $W^+_1$ we have
  \begin{equation}\label{eq:W+1-Expand} 
W^+_1  =  \sum_{r \geq 1 }\frac{ \tau_k(r)^2}{r^2} e^{-2r/Z} 
+ \sum_{\substack{r,s\geq 1\\ r\equiv   s \bmod p\\r\neq s}}
\frac{ \tau_k(r)\tau_k(s)}{rs} e^{-(r+s)/Z}.
 \end{equation}
Using the fact that $1-e^{-t} \leq t^{1/2}$ for any $t>0$ we get that 
  \begin{equation}\label{eq:W+1 MT}
  \sum_{ r \geq 1 }\frac{ \tau_k(r)^2}{r^2} e^{-2r/Z} = \sum_{r \geq 1 }\frac{ \tau_k(r)^2}{r^2} 
  + O\(Z^{-1/2}\)
  =a(k)  + O\(Z^{-1/2}\). 
   \end{equation} 
Moreover if $r\equiv  s \bmod p$ but $r\neq s$ then $\max\{r,s\} \geq p$. Hence,
  \begin{equation}\label{eq:W+1 ET} 
\sum_{\substack{r,s\geq 1\\ r\equiv    s \bmod p\\r\neq s}} 
\frac{ \tau_k(r)\tau_k(s)}{rs}e^{-(r+s)/Z}  \ll
\sum_{r\geq 1}\frac{ \tau_k(r)}{r} e^{-r/Z}  S(r), 
 \end{equation}  
 where 
 \[
 S(r) =  \sum_{\substack{s\geq 1\\ s\equiv    r  \bmod p\\s > r}} 
\frac{\tau_k(s)}{s}e^{-s/Z}.
\]
For $r \equiv 0 \bmod p$, using~\eqref{bound-tauk},  we have 
  \begin{equation}\label{eq:r = 0} 
S(r) \le  \sum_{j \geq 1}  \frac{ \tau_k(jp)}{jp} e^{-jp/Z}  
 \ll \frac{\tau_k(p)}{p} (\log 3Z/p)^k  \le p^{-1+o(1)}.
 \end{equation}  

 For $r \not \equiv 0 \bmod p$,  we define the function 
 \[
 F_r(t) = \sum_{\substack{p \le s \leq t \\ s \equiv r \bmod p}}\tau_k(s).
\]
 Using the trivial bound $\tau_k(s) = s^{o(1)}$ for $t \le p^2$ and Lemma~\ref{lem:Shiu}  for $Z^2 \ge t > p^2$
 we conclude that for the above choice of $Z$ we have 
\begin{equation}\label{boundFr}
  F_r(t) \le tp^{-1+o(1)} + \frac{t (\log t)^{k-1} }{p} \le  tp^{-1+o(1)}.
\end{equation}
 First we observe that 
\[
 S(r) =  \sum_{\substack{Z^2 \ge s\geq 1\\ s\equiv    r  \bmod p\\s > r}} 
\frac{\tau_k(s)}{s}e^{-s/Z}  + O(Ze^{-Z}). 
\] Therefore, by partial summation and the bound~\eqref{boundFr} we get 
   \begin{equation}\label{eq:Sr} 
   \begin{split} 
S(r) & 
\ll \int_{p}^{Z^2}  F_r(t) \(e^{-t/Z}\left(\frac{1}{tZ}+\frac{1}{t^2}\right)\)   \, dt   + Ze^{-Z}\\
&   \ll  \int_{p}^{Z^2}  F_r(t)   \frac{1}{t^2}   \, dt  + Ze^{-Z}\\
& \ll p^{-1+o(1)} \int_{p^2}^{Z^2} \frac{1}{t} dt  + Ze^{-Z}\ \le p^{-1+o(1)} . 
\end{split} 
\end{equation}
  
 Hence, it follows from~\eqref{eq:Sr}  that~\eqref{eq:r = 0} 
 also holds for $r \not \equiv 0 \bmod p$, which, after substituting in~\eqref{eq:W+1 ET} 
 and recalling~\eqref{bound-tauk},  implies  
 \begin{equation}\label{eq:W+1 ET-fin} 
\begin{split} 
\sum_{\substack{r,s\geq 1\\ r\equiv    s \bmod p\\r\neq s}} 
\frac{ \tau_k(r)\tau_k(s)}{rs}e^{-(r+s)/Z} &  \le
p^{-1+o(1)} \sum_{r\geq 1}\frac{ \tau_k(r)}{r} e^{-r/Z} \\
& \le p^{-1+o(1)}   (\log 3Z)^k \le p^{-1+o(1)}   . 
\end{split} 
\end{equation}

 Thus, we derive from~\eqref{eq:W+1-Expand}, \eqref{eq:W+1 MT} and~\eqref{eq:W+1 ET-fin}  that 
 with the above choice of $Z$ we have 
   \begin{equation}\label{eq:W+1 Asymp} 
W^+_1  =   
a(k)+ O\(p^{-1+o(1)}\).
 \end{equation} 
 
For $W^-_1$, we note that $r\equiv  - s \bmod p$  then $\max\{r,s\} \geq p/2$, thus the same argument gives 
   \begin{equation}\label{eq:W-1-Bound} 
W^-_1   \ll p^{-1+o(1)}.
 \end{equation}

We now turn to the   case of $\lambda \neq 1$. Notice that in this case $-\lambda$ is of order at most $2d$. 
We write
\begin{equation} 
\label{eq:W-lambda} 
W^+_\lambda  = Q^+_\lambda +  R^+_\lambda, 
 \end{equation}
where 
\begin{align*} 
&Q^+_\lambda   = \sum_{\substack{ \sqrt{p} \geq r,s\geq 1\\ r\equiv  \lambda s \bmod p}}
\frac{ \tau_k(r)\tau_k(s)}{rs} e^{-(r+s)/Z},  \\
& R^+_\lambda =  \sum_{\substack{ \max\{r,s\}> \sqrt{p} \\ r\equiv  \lambda s \bmod p}}
\frac{ \tau_k(r)\tau_k(s)}{rs} e^{-(r+s)/Z}. 
  \end{align*}

Let us first bound $R^+_\lambda$. For any fixed value of $r$, there exists a unique $s \bmod p$ such that $r \equiv \lambda s \bmod p$.  Hence, 
 \begin{align*}
\sum_{\substack{r,s\geq 1\\ r\equiv   \lambda s \bmod p\\s >\sqrt{p}}} &
\frac{ \tau_k(r)\tau_k(s)}{rs} e^{-(r+s)/Z}   \\ & \ll
\sum_{r\geq 1}\frac{ \tau_k(r)}{r} e^{-r/Z} \(\frac{\max_{s\leq p}\tau_k(s)}{\sqrt{p}}+
\sum_{\substack{s \geq p \\ \lambda s\equiv r \bmod p}} \frac{\tau_k(s)}{s} e^{-s/Z} \)  \\
& \ll p^{-1/2+o(1)}(\log 3Z)^k + p^{-1+o(1)}   \le p^{-1/2+o(1)}, 
  \end{align*} 
  where we have used again that $\tau_k(n)=n^{o(1)}$,  the bound~\eqref{eq:r = 0} to control 
  the sum over $s$ and  the bound~\eqref{bound-tauk}   to control the sum over $r$.   By symmetry, we can bound similarly the part of the sum where $r>\sqrt{p}$ and thus obtain
    \begin{equation}\label{eq: R lambda} 
R^+_\lambda \le  p^{-1/2+o(1)}.
 \end{equation}

 We now focus on the sum $Q^+_\lambda$. Let $(r_0,s_0)$ with  $r_0 \equiv \lambda s_0 \bmod p$ be  such that
\[
\rho(\lambda,p)=   |r_0s_0|.
\]
Clearly, we have 
$\gcd(r_0,s_0)=1$ as otherwise the product $r_0s_0$ would not be of smallest absolute value among all the solutions. Assume that there is another pair  $(r,s)$ such that $r \equiv \lambda s \bmod p$, which also contributes to  $Q^+_\lambda$, that is,  we have 
\begin{equation}\label{eq:smallrs} 
1\leq r,r_0,s,s_0 < \sqrt{p}.
 \end{equation}
 Thus we have
\[ rs_0 \equiv r_0s \bmod p\]
 which under the inequalities~\eqref{eq:smallrs} implies that $rs_0 = r_0 s$.
Solving the equation, we get $r = \delta r_0, s= \delta s_0$ for some integer $\delta \geq 1$. 
Hence,
\begin{equation}
\begin{split}
\label{eq: Q lambda} 
Q^+_\lambda & \ll \sum_{\substack{ 1 \leq r,s\leq \sqrt{p}\\ r\equiv  \lambda s \bmod p}}
\frac{ \tau_k(r)\tau_k(s)}{rs}  \\
&   \ll \sum_{1\leq \delta \leq \sqrt{p}\min(1/r_0,1/s_0)} \frac{ \tau_k(\delta r_0)\tau_k(\delta s_0)}{\delta^2 r_0s_0} \\
& \ll \frac{\tau_k(r_0s_0)}{r_0s_0} \sum_{\delta \geq 1} \frac{\tau_k(\delta)^2}{\delta^2}  \ll \frac{\tau_k(r_0s_0)}{r_0s_0} 
 \ll  \rho(\lambda,p)^{-\eta}    
 \end{split}
\end{equation}
for any constant $\eta < 1$.  
 Substituting~\eqref{eq: R lambda} and~\eqref{eq: Q lambda} in~\eqref{eq:W-lambda}, and treating
 $W^-_\lambda$ in a similar way,  we obtain 
 \begin{equation} 
\label{eq:W lambda bound} 
W^\pm_\lambda  \ll   \rho(\pm \lambda,p)^{-\eta}    +   p^{-1/2+o(1)}.
 \end{equation}
 Therefore 
 \begin{equation}\label{upp-bndW}
\sum_{\substack{\lambda \in \cG_m\\\lambda \ne 1}} W^\pm_\lambda  \ll  \sR_{\eta} (m,p)   + O\(d p^{-1/2+o(1)}\)  .
\end{equation}

Recall that by Lemma~\ref{lem:rho2 - extreme}, we have $\vartheta(m,p) \gg p^{1/\varphi(d)}$.
Now, without loss of generality we can assume that 
\begin{equation}\label{eq:small d} 
 \varphi(d)=o(\log p)
\end{equation}
as otherwise the result is trivial. 

Next, from the bound~\eqref{upp-bndW}, the assumption~\eqref{eq:small d} and Lemma~\ref{lem:Sum rho}
(chosen with $\alpha = \eta$ and an arbitrary small $\beta$) we obtain
 \begin{equation}\label{eq:Wpm fin} 
\sum_{\substack{\lambda \in \cG_m\\\lambda \ne 1}} W^\pm_\lambda   \ll  \vartheta(m,p)^{-\kappa} + p^{-\eta/4+o(1)} \ll p^{-\kappa/\varphi(d)} + p^{-\eta/4+o(1)} 
\end{equation}
for any constant $\kappa < \eta$,  where we  have used Lemma~\ref{lem:rho2 - extreme} in the last step.

We observe that since $\eta< 1$ is arbitrary, then so is $\kappa<\eta$. Substituting the asymptotic formula~\eqref{eq:W+1 Asymp}  and 
 the bounds~\eqref{eq:W-1-Bound}  and~\eqref{eq:Wpm fin} 
 in~\eqref{eq:Mk W}  and~\eqref{eq:Mk W-}, we derive the desired result.

 \subsection{Proof of Theorem~\ref{thm:almostall}}  
  We present the proof only in the case of $M_{2k}^{-}(p,m)$ as the proof clearly 
 extends without any changes to the moments $M_{2k}(p,m)$.
 
 We proceed as in the proof of Theorem~\ref{theo-kth} and arrive to~\eqref{eq:Mk W}.
 
 Let 
\begin{equation}\label{eq:set F} 
 \cF = \bigcup_{i = 0}^I  \cE(D,2^i, 2^iR)
\end{equation}
where $I$ is the smallest integer with $2^I \ge D$. 
By Lemma~\ref{lem: almost all} we see that 
\begin{equation}\label{eq:card F} 
\#\cF \ll D^2   R (\log R)^2 \log Q. 
\end{equation}
Hence it remains to obtain asymptotic formulas for $M_{2k}(p,m)$ and  $M_{2k}^-(p,m)$
for primes $p \notin \cF$.

For $W_1^{\pm}$ we simply use~\eqref{eq:W+1 Asymp} and~\eqref{eq:W-1-Bound}. 

Next, for $\lambda \in \cG_m \setminus \{1\}$ we recall that~\eqref{eq:W lambda bound} gives for any constant $\kappa<1$
\[
 W^\pm_\lambda  \ll   \rho(\lambda,p)^{-\kappa}    +   p^{-1/2+o(1)}.  
\] 
 Hence, by the definition of $\cF$, for $p \notin \cF$ we have 
\begin{align*}
\sum_{\substack{\lambda \in \cG_m\\\lambda \ne 1}} W^\pm_\lambda &  \ll Dp^{-1/2+o(1)} + 
\sum_{i=0}^I  \sum_{\substack{\lambda \in \cG_m\\\lambda \ne 1\\  2^{i-1} R < \rho(\lambda,p) \le 2^i R }} 
\(2^{i } R\)^{-\kappa}  \\
 &  \ll   Dp^{-1/2+o(1)} + 
\sum_{i=0}^I  2^i \(2^{i } R\)^{-\kappa} \ll D^{1-\kappa} R^{-\kappa} +  Dp^{-1/2+o(1)}, 
\end{align*}
which concludes the proof.

 \subsection{Proof of Corollary~\ref{cor:Conj Almost All}}  
Let $0<\varepsilon<1$ and define
$$
D = Q^{1/2-\varepsilon} \mand R =  Q^{\varepsilon}.
$$
Hence taking $\kappa=1-\varepsilon$ in Theorem~\ref{thm:almostall} we have 
$$
M_{2k}(p,m), M_{2k}^-(p,m)  =  a(k) + O(D^\varepsilon R^{-1+\varepsilon}+Dp^{-1/2+o(1)})
$$
under the relevant conditions on $d$,
for all except at most $E$ primes   $p \in [Q, 2Q]$, 
where 
$$
E  \ll  \frac{D^2 R (\log R)^2}{\log Q}   \ll Q^{1-\varepsilon+o(1)}.
$$
Clearly,  for a sufficiently large $Q$,  we obtain  $E\le Q^{1-\varepsilon/2}$. 

It remains to observe that 
$$
Dp^{-1/2+o(1)} \le Q^{- \varepsilon + o(1)} \le  Q^{-\varepsilon/2}
$$
and 
$$
D^\varepsilon R^{-1+\varepsilon} =  Q^{(1/2-\varepsilon) \varepsilon  -  \varepsilon(1-\varepsilon)}
=  Q^{-\varepsilon/2}. 
$$

 \section{Proofs of results on moments and non-vanishing of $L(1/2,\chi)$}
\label{sec:Mom-Nonvanis}

\subsection{Non-vanishing of $L(1/2,\chi)$ inside a subgroup}
In this section the implied constant may depend only  $\varepsilon> 0$
but are uniform with respect to $k$. 

For simplicity we only consider even Dirichlet characters, the case of odd characters being similar. Let us briefly summarize the idea of constructing mollifiers. To each even character $\chi$, we associate the function 
\[  
M(\chi)=\sum_{h \leq H} \frac{x_h \chi(h)}{\sqrt{h}}, 
\]
where $X=(x_h)$ is a sequence of real numbers supported on $1\leq h \leq H$.

 The purpose of the function $M(\chi)$ is to ``mollify'' the large values of $L(1/2,\chi)$ on average over $\chi \in \cX_p$. Our aim is to employ this strategy on average over subgroups of characters.
We now consider the mollified moments 
\begin{align*}
& \mathcal{C}_m(H)= \sum_{\chi \in \cX_{p,m}^+}  M(\chi)L(1/2, \chi),  \\
& \mathcal{D}_m(H) = \sum_{\chi \in  \cX_{p,m}^+}  |M(\chi)L(1/2, \chi)|^{2}.
\end{align*} 
The  Cauchy--Schwarz inequality implies that
\begin{equation}
\label{eq: Cauchy}
\sum_{\substack{\chi \in \cX_{p,m} \\ L(1/2,\chi) \neq 0}} 1
\ \ge \sum_{\substack{\chi \in \cX_{p,m}^+ \\ L(1/2,\chi) \neq 0}} 1 \ \ge \ \frac{\mathcal{C}_m(H)^2}{\mathcal{D}_m(H)}. 
\end{equation}

In the full group case, Iwaniec and Sarnak~\cite{IS} showed that maximising the ratio $\frac{\mathcal{C}_m(H)^2}{\mathcal{D}_m(H)}$ with respect to the vector $X=(x_h)$, the optimal mollifier is approximately given by
\[
\widetilde{M}(\chi)=\sum_{h \leq H} \frac{\mu(h) \chi(h)}{\sqrt{h}}\left(1-\frac{\log h}{\log H}\right),
\]
From this, they deduce that there are more than $(1/3+o(1))p$ characters $\chi$ modulo $p$ 
such that $L(1/2,\chi) \neq 0$. Let us mention that more complicated mollifiers can be constructed 
in order to improve the proportion of non-vanishing values  of $L(1/2, \chi)$, see~\cite{Bui,BPZ,KN1,KMN2,Pratt}. However, we emphasise that our aim here is to handle the smallest possible subgroups in a simple way rather than optimise this proportion.

 We use the mollifiers constructed in~\cite{IS} but average it over a thinner set given by the subgroup $\cX_{p,m}$. 
 Hence, as in~\cite{IS} we assume from now on that the coefficients of $M(\chi)$ satisfy
 \begin{equation}\label{eq:conditions}x_1=1 \text{ and } x_h \ll 1. \end{equation}  
 The error terms involved in the computation of the asymptotic formulas for  $\mathcal{C}_m(H)$ and $\mathcal{D}_m(H)$  forces us to reduce the length of the mollifier compared to~\cite{IS}. We show that the length of this mollifier  depends crucially on 
 $\vartheta(m,p)$, the smallest solution to the linear congruences $r \equiv \lambda s \bmod p$.

In order to evaluate $\mathcal{C}_m(H)$ and $\mathcal{D}_m(H)$,  we introduce
\begin{align*}
&\mathcal{A}_m^+(h)= \sum_{\chi \in \cX_{p,m}^+}  \chi(h)L(1/2, \chi), \\
&\mathcal{B}_m^+(h,k)=\sum_{\chi \in \cX_{p,m}^+} \chi(h)\overline{\chi}(k)\vert L(1/2,\chi)\vert^{2}.
\end{align*}
Our goal is to evaluate $\mathcal{A}_m^+(h)$ and $\mathcal{B}_m^+(h,k)$ for paramaters $h,k$ as ``large'' as possible.
We recall the notations $\Gamma(z)$ and $\gamma$ from Section~\ref{sec:L1/2}.
Our first technical  result is the following.

\begin{lemma}\label{asymp-A}
 For any integer $h$ such that $\gcd(h,p)=1$  we have
\[
\mathcal{A}_m^+(h) =  \frac{m\delta(h)}{2} +O\( m\sqrt{h} \(p^{-1/2+o(1)}+\sR_{1/2} (m,p)\)\), 
\] 
where $\delta(h)=1$ if $h=1$ and $\delta(h)=0$  otherwise. 
\end{lemma}

\begin{proof} Let $\varepsilon>0$. By~\cite[Lemma~3.2]{Bui}, we have for an even character $\chi$ and any $A>0$
\begin{equation}\label{approxL}
L(1/2,\chi)=\sum_{n \geq 1} \frac{\chi(n)}{\sqrt{n}}V\left(\frac{n}{p^{1+\varepsilon}}\right) + O(p^{-A}),
\end{equation}
for some function $V(x)$ which satisfies
   \begin{equation}\label{eq: func V} 
V(x)  = 1 +O \(x^{1/2 + o(1)}\) \mand V(x) \ll x^{-c} 
\end{equation}
for any fixed $c > 0$, 
(that is, $V(x)$ is close to $1$ for small $x$ and decays as $x^{-c}$ for $x \to \infty$). 

Averaging over $\chi \in \cX_{p,m}$ and using~\eqref{eq: even m} and~\eqref{approxL}, we obtain 
\begin{equation}\label{eq:A}
\mathcal{A}_m^+(h) = \frac{m}{2}\sum_{\lambda \in \cG_m} \sum_{hn \equiv \pm \lambda \bmod p}n^{-1/2}V(n/p^{1+\varepsilon}) + O(p^{-A}).
\end{equation} The diagonal terms $hn=1$ give the main contribution to $\mathcal{A}_m^+(h)$ 
and is equal by~\eqref{eq: func V} to
\begin{equation}\label{mainA}
\frac{m\delta(h)}{2} + O(1). 
\end{equation}
The other terms contribute at most
\[
\frac{m}{2}\sum_{\lambda \in \cG_m} 
\sum_{\substack{hn \equiv \pm \lambda \bmod p \\ hn \neq 1}}n^{-1/2}V(n/p^{1+\varepsilon}).
\]

Let now $\lambda=1$. Then we see that $n>p$ runs 
through a union of two arithmetic progressions of the kind $a+j p$, $j =1, 2, \ldots$, with $a > 0$.
Therefore, using the estimate $V(x) \ll (1+\vert x\vert)^{-1}$, we get 
\begin{equation}\label{eq:errorA1}
\begin{split}
 \sum_{\substack{hn \equiv \pm  1 \bmod p \\ hn \neq 1}} n^{-1/2}&V(n/p^{1+\varepsilon}) \\
 & \ll  \sqrt{h/p}+\sum_{r =1}^\infty    \sum_{j =1}^\infty  \frac{1} {\sqrt{jp}}  V\( j/p^{\varepsilon} \)  \\
 & \ll   \sqrt{h/p}+p^{-1/2+\varepsilon}  \sum_{j =1}^\infty  \frac{1} {j^{3/2} }  \ll p^{-1/2+o(1)}\sqrt{h}.
   \end{split}
   \end{equation}
Similarly for any $\lambda \neq 1$, we have
\begin{equation}\label{errorA2} \sum_{hn \equiv \pm \lambda \bmod p}n^{-1/2}V(n/p^{1+\varepsilon})
 \ll \sqrt{h/\lambda}.\end{equation} Hence by~\eqref{eq:errorA1} and~\eqref{errorA2}, we obtain
\begin{equation}
\begin{split}
 \label{error-Afinal}
\sum_{\lambda \in \cG_m} \sum_{\substack{hn \equiv \pm \lambda \bmod p \\ hn \neq 1}}
& \frac{1}{\sqrt{n}}V(n/p^{1+\varepsilon})\\
  &\ll \sqrt{h} \left(p^{-1/2+o(1)} + \sum_{1 \neq \lambda \in \cG_m}\frac{1}{\sqrt{\lambda}}\right)  \\
& \ll \sqrt{h} \left(p^{-1/2+o(1)}+\sR_{1/2} (m,p)\right), 
   \end{split}
   \end{equation} 
where we have use the trivial inequality $\lambda \geq \rho(\lambda,p)$. The result follows combining~\eqref{eq:A}, \eqref{mainA} and~\eqref{error-Afinal}.
\end{proof} 

In particular, applying Lemma~\ref{asymp-A}  for $h=1$, together with Lemmas~\ref{lem:rho2 - extreme} and~\ref{lem:Sum rho}, we get that
\[ 
\frac{2}{m}\sum_{\chi \in \cX_{p,m}^+} L(1/2, \chi) = 1+o(1)
\] 
holds in the range $\varphi(d)=o(\log p)$.  

We now turn our attention to $\mathcal{B}_m^+(h,k)$. We recall by~\cite[Lemma~2]{Sound2} that for an even 
character $\chi$ we have 
\[
 \vert L(1/2,\chi)\vert^2 = 2  A(\chi) , 
\]
where 
\[
A(\chi) = \sum_{r,s=1}^\infty \frac{\chi(r) \overline\chi(s)} {\sqrt{rs}} W\(\frac{\pi rs}{p}\)
\]
for some function $W(x)$ which satisfies~\eqref{eq: func V}.  We note that a similar relation also holds in the case of odd characters.

 We define $L$ such that
\begin{equation}\label{def:L}
\log L= \frac{1}{2}\log \frac{p}{\pi} +\frac{1}{2}\frac{\Gamma'}{\Gamma}\left(\frac{1}{4}\right)+\gamma.
\end{equation}

We recall the definition of $\sR_\alpha(m,p)$ from~\eqref{eq:  Rmp}.  
 
\begin{lemma}\label{asymp-Bm}
 For any integers $h,k$  such that $\gcd(h,k)=\gcd(hk,p)=1$  we have
\[
\mathcal{B}_m^+(h,k)= \frac{m(p-1)}{2p\sqrt{hk}} \log  \frac{L^2}{hk} + O(\Delta), 
\]
where 
\[
\Delta = m\(\sqrt{hk} \sR_{1/2}(m,p)
 \log p + d p^{-1/6 + o(1)} +  k^{3/2} p^{-1/2} \). 
\]
\end{lemma}

\begin{proof}  
Similarly as in~\cite[Equation~(3.8)]{IS}, we derive using~\eqref{eq: even m} that 
\[  
\mathcal{B}_m^+(h,k)= m\sum_{\lambda \in \cG_m} \  \sideset{}{^*} \sum_{hr \equiv \pm ks \lambda \mod p} (rs)^{-1/2}W(\pi rs/p), 
\] 
where the $\Sigma^*$ means that the summation is restricted to numbers relatively prime to $p$. The main contribution to 
$\mathcal{B}_m^+(h,k)$ comes from the diagonal terms $hr=ks$ (occurring when $\lambda=1$). Assuming $\gcd(h,k)=1$, we see as in~\cite{IS} that this contribution $\mathcal{B}_m^0(h,k)$ is
\[  
m\,  \sideset{}{^*} \sum_{hr = ks} (rs)^{-1/2}W(\pi rs/p)= \frac{m}{\sqrt{hk}}S\left(\frac{p}{\pi hk}\right),
\] 
where
\[  
S(X)= \sum_{(n,p)=1}n^{-1}W(n^2/X).
\] 
Using~\cite[Lemma~4.1]{IS}, we see that 
\begin{equation}\label{def:diagterms}
\mathcal{B}_m^0(h,k)= \frac{m(p-1)}{2p\sqrt{hk}} \log  \frac{L^2(p)}{hk} +  O(mp^{-1/2}).
\end{equation} We now turn to the error terms. Applying~\eqref{eq: func V} we have 
\begin{align*}
\left| \sum_{\substack{r,s=1\\ rs > Z}}^\infty \frac{\chi(r) \overline\chi(s)} {\sqrt{rs}} W\(\frac{\pi rs}{p}\)\right|
& \ll  \sum_{u > Z}^\infty \frac{\tau(u)} {\sqrt{u}} W\(\frac{\pi u}{p}\)\\
&\ll  p^2 \sum_{u > Z}^\infty \frac{\tau(u)} {u^{1/2+ 2}} \le   p^2 Z^{-3/2+o(1)}.
\end{align*}

We take $Z$ to be the largest power of $2$ which is less than $p^2$. In particular 
\[
p^2 \ll Z  < p^{2}
\]  and we obtain 
\[  
A(\chi) =  \sum_{\substack{r,s=1\\ rs \le  Z}}^\infty  \frac{\chi(r) \overline\chi(s)} {\sqrt{rs}} W\(\frac{\pi rs}{p}\)
+O\(p^{-1+o(1)}\). 
\] 
Therefore,  
\[
\mathcal{B}_m^+(h,k)  =  \mathcal{B}_m^0(h,k) + \frac{m}{2}  \sum_{\lambda \in \cG_m} 
\( S^+(\lambda) +  S^-(\lambda) \) + O(1), 
\]
where for $\varepsilon\in \{+,-\}$ we have
\[
 S^{\varepsilon}(\lambda) =  \sum_{\substack{rs \le  Z\\hr \equiv \varepsilon \lambda ks \bmod p \\ hr \neq ks}}  \frac{1} {\sqrt{rs}}  W\(\frac{\pi rs}{p} \).
\] First, let $\lambda=1$. Then, we proceed similarly to the derivation of~\eqref{eq:errorA1}   and~\eqref{errorA2}.
Namely, we can always assume that $hr < ks$ and in particular for a given $r$ we see that $s> p$ runs 
through an arithmetic progression of the kind $a+j p$, $j =1, 2, \ldots$, with $a > 0$. 
Therefore, using~\eqref{eq: func V} with $c = 1$, 
we derive 
\begin{equation}\label{eq:sum S-1}
\begin{split}
 S^{\varepsilon}(1) & \ll  \sum_{r =1}^\infty \frac{1}{\sqrt{r}} \left\{\sqrt{k/p} W(\pi r/k) +   \sum_{j =1}^\infty  \frac{1} {\sqrt{jp}}  W\( \pi r j \) \right\} \\
 & \ll  \frac{k^{3/2}}{\sqrt{p}} \sum_{r =1}^\infty \frac{1} {r^{3/2} }    + \frac{1} {\sqrt{p}}  \sum_{r =1}^\infty    \sum_{j =1}^\infty  \frac{1} {(rj)^{3/2} } \\
 &  \ll k^{3/2} p^{-1/2}. 
   \end{split}
   \end{equation} We now consider the case of   $\lambda\in  \cG_m \setminus \{1\}$. Since $S^{-}(\lambda)$ can be treated in the same manner, we concentrate on $ S^+(\lambda)$. 

We first split the summation over positive integers $r, s$ with $rs < Z$ into $O\(\log Z)^2\)$ ranges 
of the form $r \sim z_1$, $s \sim z_2$ for some   $z_1, z_2 \ge 1$ which run through powers of $2$ with $z_1z_2 \ll Z$. 

First we consider the range in which we have
   \begin{equation}\label{eq: large z_1z_2} 
z_1z_2 \ge p/2.
\end{equation}

We now  observe that for any solution $(r,s)$ in positive integers to the equation $ hr\equiv \lambda ks \bmod p$ 
with $\delta = \gcd(r,s)$ we have either $p \mid \delta$, which is impossible for $rs < Z$ or
$hr/\delta \equiv \lambda ks/\delta \bmod p$ and thus $|hkrs|/\delta^2 \ge \rho(\lambda,p)$. Hence, 
\[
\delta \le \sqrt{hkrs/  \rho(\lambda,p)}  \le  \sqrt{4hkz_1z_2/  \rho(\lambda,p)} .
\]
Then,  writing $r = u\delta$, $s =v\delta$, we have 
\begin{equation}\label{eq: rs uv1} 
\begin{split}
 &\sum_{\substack{r\sim z_1\\ s \sim z_2 \\ hr \equiv \lambda ks \bmod p}}  \frac{1} {\sqrt{rs}}  W\(\frac{\pi rs}{p} \)\\
& \qquad \quad \ll \sum_{1 \le \delta \le  \sqrt{4hkz_1z_2/  \rho(\lambda,p)} }
  \sum_{\substack{u\sim z_1/\delta \\ v \sim z_2/\delta \\ \gcd(u, v) =1\\ hu \equiv \lambda kv \bmod p}}  \frac{1} {\sqrt{uv} \delta}  W\(\frac{\pi uv\delta^2}{p} \).
   \end{split}
   \end{equation}

Next, from the bound~\eqref{eq: func V}  (with $c =1$) and Lemma~\ref{lem: ModHyperb}, 
with $z_1$ and $z_2$ as in~\eqref{eq: large z_1z_2}, we now derive 
\begin{equation}\label{eq: rs uv2} 
\begin{split}
 \sum_{\substack{r\sim z_1\\ s \sim z_2 \\ hr \equiv \lambda ks \bmod p}} & \frac{1} {\sqrt{rs}}  W\(\frac{\pi rs}{p} \)\\
&  \ll \sum_{1 \le \delta \le  \sqrt{4hkz_1z_2/  \rho(\lambda,p)} }
\frac{1} {\sqrt{z_1 z_2} } \(1+   \frac{z_1 z_2} {p\delta^2}\)     \frac{p} {z_1 z_2}   \\
& \ll   \sum_{1 \le \delta \le  \sqrt{4hkz_1z_2/  \rho(\lambda,p)} }
\(  \frac{p} {(z_1 z_2)^{3/2} } +    \frac{1} {(z_1 z_2)^{1/2} \delta^2}\)\\
& \ll  \(  \frac{p\sqrt{hk}} {z_1 z_2 \sqrt{ \rho(\lambda,p)}  } +    \frac{1} {(z_1 z_2)^{1/2} }\) 
\ll  \frac{p\sqrt{hk}} {z_1 z_2 \sqrt{ \rho(\lambda,p)}  }  . 
\end{split}
\end{equation}
Now,  summing over  $z_1$ and $z_2$ running through the powers of $2$ and satisfying~\eqref{eq: large z_1z_2}
we obtain that the total contribution $\fS_1(\lambda)$ to $S^{+} (\lambda)$ from all $r \sim z_1$, $s \sim z_2$ with such $z_1$ and $z_2$ 
is bounded by
\begin{equation}\label{eq: Range p p2} 
\fS_1(\lambda) \ll  \frac{\log p} {  \sqrt{ \rho(\lambda,p)} }\sqrt{hk}. 
\end{equation}

Now we look at the  range in which we have
   \begin{equation}\label{eq: small z_1z_2} 
z_1z_2  < p/2.
\end{equation}
We note that in this regime $\max\{r,s\} \le \max\{2z_1, 2z_2\} < p$. Hence the congruence 
$hr \equiv \lambda ks \bmod p$ defines a unique $s$ for each $r$. 

We first consider the case when $r < s$. 
 We now assume  $s > p^{2/3}$  and estimate   the total  contribution $\fS^{<}_{2}$
 from such terms over all $z_1$ and $z_2$ with~\eqref{eq: small z_1z_2}. In this range, we have $r \le p/(2s) \le  p^{1/3}$ (making the condition $r\le s$ redundant). 

Moreover by~\eqref{eq: func V}, we have the bound
   \begin{equation}\label{eq: W0 < 1} 
 W\(\frac{\pi rs}{p} \)\ll 1.
\end{equation}
 
  Hence, we get 
\begin{equation}\label{eq: large s} 
\begin{split}
\fS^{<}_{2} (\lambda) & \ll \sum_{\substack{hr \equiv \lambda ks \bmod p\\ rs < p/2\\ r < s \\ s > p^{2/3}}}  
 \frac{1} {\sqrt{rs}}  W\(\frac{\pi rs}{p} \) \\
 & \ll  p^{-1/3} \sum_{r \le p^{1/3}}    \frac{1} {\sqrt{r} } \ll p^{-1/6}.
   \end{split}
\end{equation}

Next, we estimate  the contribution $\fS^{<}_{3}$ from the terms with $s  < p^{2/3}$ (again in 
total over all $z_1$ and $z_2$ with~\eqref{eq: small z_1z_2}). 
We fix  positive integers $r_0$ and $s_0$ with 
\[
r_0 \equiv \lambda s_0 \bmod p \mand r_0s_0 =  \rho(\lambda k/h,p).
\]  This is not hard to see that
\begin{equation}\label{bound:rhohk}
\rho\left(k\lambda/h,p\right) \geq \frac{\rho(\lambda,p)}{hk}.
\end{equation}
Then for any other solution to $r \equiv \lambda s \bmod p$ we have 
\[
r_0s \equiv r s_0 \bmod p.
\]
Using $r \le s < p^{2/3}$ we see that for $  \rho(\lambda k/h,p) \le p^{1/3}$ the above congruence
implies the equation $r_0s = r s_0$ and since the minimality of $(r_0,s_0)$ implies $\gcd(r_0,s_0)=1$
we see that 
\[r =  r_0 \delta \mand s = s_0  \delta\]
for some $\delta < p$. Hence, in this case we get using~\eqref{bound:rhohk}
\begin{equation}\label{eq: small s1} 
\begin{split}
\fS^{<}_{3} (\lambda) &\ll \sum_{\substack{r\sim z_1\\ s \sim z_2 \\ hr \equiv \lambda ks \bmod p\\ rs < p/2 \\ r < s  \le p^{2/3}}}  
 \frac{1} {\sqrt{rs}}  W\(\frac{\pi rs}{p} \) \\
&  \ll   \frac{\sqrt{hk}} {\sqrt{ \rho(\lambda,p)} }   \sum_{ \delta =1 }^p  \frac{1} {\delta}  \ll  \frac{\log p} {\sqrt{ \rho(\lambda,p)} }\sqrt{hk} . 
 \end{split}
\end{equation}
On the other hand, if $\rho(\lambda k/h,p) > p^{1/3}$, then we proceed as in~\eqref{eq: rs uv1}  and~\eqref{eq: rs uv2} 
with the only difference that in this case we have~\eqref{eq: W0 < 1}, while bound of  Lemma~\ref{lem: ModHyperb}, 
becomes
\[
1+   \frac{z_1 z_2} {p\delta^2} \ll 1.
\]
Hence
\begin{align*} 
\sum_{\substack{r\sim z_1\\ s \sim z_2 \\ hr \equiv \lambda ks \bmod p\\ rs < p/2 \\ r < s  \le p^{2/3}}}  
 \frac{1} {\sqrt{rs}}  W\(\frac{\pi rs}{p} \) 
 & \ll    \sum_{1 \le \delta \le  \sqrt{ z_1z_2/  \rho(\lambda k/h,p)} } \frac{1} {\sqrt{ z_1 z_2}} \\
 & \ll  \frac{1} {\sqrt{ \rho(\lambda k/h,p)}}. 
\end{align*} 

Summing over all choices of $z_1$ and $z_2$  running through the powers of $2$ and 
satisfying~\eqref{eq: small z_1z_2}, we obtain 
\begin{equation}\label{eq: small s2} 
\fS^{<}_{3} (\lambda) \ll    \frac{(\log p)^2} {\sqrt{ \rho(\lambda k/h,p)} }  \le p^{-1/6 + o(1)}
\end{equation}
by our assumption that $\rho(\lambda k/h,p) > p^{1/3}$. 

We also note that $\lambda \in \cG_m$ is equivalent to $\lambda^{-1}  \in \cG_m$ and 
in fact $ \rho(\lambda,p) = \rho\(\lambda^{-1},p\)$. Hence for the corresponding 
contributions $\fS^{>}_{2} (\lambda)$ and $\fS^{>}_{3} (\lambda)$ from the terms with $hr>ks$ we have full 
analogues of the bounds~\eqref{eq: large s}, \eqref{eq: small s1} and~\eqref{eq: small s2}.  
 
 Thus, collecting   bounds~\eqref{eq: Range p p2}, \eqref{eq: large s},  \eqref{eq: small s1} 
and~\eqref{eq: small s2}  we obtain 
\[
S^+(\lambda) \ll   \sqrt{hk}\frac{\log p } {\sqrt{ \rho(\lambda,p)} } + p^{-1/6 + o(1)}.
\]
We also have an identical bound for $ S^{-}(\lambda)$,  that is, we can summarise this as
\begin{equation}\label{eq: S_a^eps} 
 S^{\pm}(\lambda) \ll  \sqrt{hk} \frac{\log p } {\sqrt{ \rho(\lambda,p)} } + p^{-1/6 + o(1)}.
\end{equation} Summing over $\lambda \in \cG_m$ and combining~\eqref{def:diagterms},  \eqref{eq:sum S-1} and~\eqref{eq: S_a^eps}, we conclude the proof.
\end{proof}

\subsection{Proof of Theorem~\ref{thm:secondmoment}}

It follows from Lemma~\ref{asymp-Bm} applied in the case $h=k=1$ that 
\begin{align*} \frac{2}{m}\sum_{\chi \in \cX_{p,m}^+} \vert L(1/2,\chi)\vert^2 & = \frac{2}{m}\mathcal{B}_m^+(1,1) \\
& =  \log L^2 + O\(\sR_{1/2}(m,p)\log p + dp^{-1/6 + o(1)} \), 
\end{align*}
where as before $L$ is given by~\eqref{def:L}. 
Then the result follows from  Lemma~\ref{lem:Sum rho} 
(applied with $\alpha=1/2$ and an arbitrary small $\beta$) and Lemma~\ref{lem:rho2 - extreme}.
 
\subsection{Proof of Theorem~\ref{thm: nonvanishing}}

By Lemma~\ref{asymp-A}, we obtain that
\begin{equation}
\begin{split} 
\label{eq:asympC}
\mathcal{C}_m(H)  & = \sum_{h\leq H} \frac{x_h}{\sqrt{h}}\sum_{\chi \in \cX_{p,m}^+} \chi(h)L(1/2,\chi)   \\
&= \frac{m}{2} +  O\left( mH \left(p^{-1/2+o(1)}+\sR_{1/2} (m,p)\right)\right), 
\end{split} 
\end{equation}
where we have used~\eqref{eq:conditions}.

 We derive an analogue of~\cite[Equation~(4.8)]{IS},
\begin{equation}
\begin{split}
\label{average-errorhk}&\sum_{h,k \leq H} x_h  x_k \left( \sR_{1/2} (m,p) \log p
+ \frac{d}{\sqrt{hk}}p^{-1/6 + o(1)} +  \frac{k}{\sqrt{h}} p^{-1/2} \right)  \\
& \qquad \qquad \ll  H^2  \sR_{1/2} (m,p) \log p+ dHp^{-1/6+o(1)} + H^{5/2}p^{-1/2},
\end{split} 
\end{equation}
where we have again  used~\eqref{eq:conditions}.

Hence by Lemma~\ref{asymp-Bm} and the bound~\eqref{average-errorhk} we see that 
\begin{equation}
\begin{split} 
\label{eq:asympD}
\mathcal{D}_m(H) & =  \sum_{h,k \leq H} \frac{x_h x_k}{\sqrt{hk}}\sum_{\chi \in \cX_{p,m}^+} \chi(h)\overline{\chi}(k)|L(1/2, \chi)|^{2}   \\
& = \frac{m(p-1)}{2p} \cQ(X)  + O \Bigl( m  \bigl(H^2  \sR_{1/2} (m,p)  \log p  \\ 
& \qquad \qquad  \qquad \qquad  \qquad \ + dHp^{-1/6+o(1)} + H^{5/2}p^{-1/2}\bigr)\Bigr)  , 
\end{split} 
\end{equation}
where $ \cQ(X)$ is as in~\cite[Section~5]{IS} and  the coefficients $(x_h)$ are chosen as in~\cite[Lemma~6.1]{IS} such that
\begin{equation}\label{eq:defquad}
\cQ(X) = \frac{p}{p-1} \left\{1+\frac{\log p}{\log H} + O\left(\frac{\log \log p}{\log H}\right)\right\}.
\end{equation} 
Let us emphasise that the treatment of the main term  in~\eqref{eq:asympD} is completely identical 
to that in~\cite{IS}.
We now assume that $d\leq p^{1/8-\varepsilon}$ and choose $H= \vartheta(m,p)^{\varepsilon/4}$. 
Hence by Lemma~\ref{lem:Sum rho} (with $\alpha=1/2$, $\beta=1/8-\varepsilon$ and 
$\kappa =2\varepsilon$) we get
\[
H^2 \sR_{1/2} (m,p)\ll \vartheta(m,p)^{\varepsilon/2} \(\vartheta(m,p)^{-2\varepsilon} + p^{-\varepsilon}\) 
\ll  \vartheta(m,p)^{-\varepsilon/2}
\]   
where we have used the trivial bound $\vartheta(m,p) \leq p$. Hence
\begin{equation}  \label{errorD_m}
H^2  \sR_{1/2} (m,p)  \log p   \ll   \vartheta(m,p)^{-\varepsilon/2}  \log p
= o\left( \frac{\log p}{\log H}\right).   
 \end{equation}  
Noticing that $H \ll p^{\varepsilon/4}$, we see that the other error terms in~\eqref{eq:asympD} are negligible.
 Thus, we deduce from~\eqref{eq:asympC}, \eqref{eq:asympD}, \eqref{eq:defquad} and~\eqref{errorD_m} and that
\begin{equation}
\begin{split}
\label{eq: C and D}
&\mathcal{C}_m(H) = 
\(\frac{1}{2} + o(1)\) m,  \\
 &  \mathcal{D}_m(H) = \(\frac{1}{2} + o(1)\)  m \(1 + \frac{\log p}{\log H}\),
\end{split} 
\end{equation}
as $p\to \infty$. 
Hence, using the Cauchy--Schwarz inequality and  recalling the choice of $H$,  we infer
from~\eqref{eq: Cauchy} and~\eqref{eq: C and D}
\[
\sum_{\substack{\chi \in \cX_{p,m} \\ L(1/2,\chi) \neq 0}} 1 \geq \frac{\mathcal{C}_m(H)^2}{\mathcal{D}_m(H)} \gg 
\frac{\log \vartheta(m,p)}{\log p} \sum_{\chi \in \cX_{p,m}} 1,
\]
which concludes the proof.

\subsection{Proof of Theorem~\ref{thm:almostall 1/2}}

 We follow the proof of Theorem~\ref{thm: nonvanishing} and define
$D=p^{\alpha}$, $H=p^{\beta}$ and $R=p^{\eta}$. 
As in the proof of Theorem~\ref{thm:almostall} consider the set $\cF$ given by~\eqref{eq:set F}. 
We also recall the bound~\eqref{eq:card F} on its cardinality, which  with above parameters gives 
$$
\#\cF \le Q^{2\alpha + \eta+ o(1)}. 
$$

We  also see that for $p \not \in \cF$ the same computation as in the proof of Theorem~\ref{thm:almostall}, 
used with $\kappa = 1/2$ yields
$\sR_{1/2} (m,p) \ll  D^{1/2}R^{-1/2}$.
Hence,
\[
 H^2 (\log p) \sR_{1/2} (m,p)+ dHp^{-1/6+o(1)} + H^{5/2}p^{-1/2} \ll p^{\omega+o(1)}, 
\]
 where 
\[
\omega = \max(2\beta+\alpha/2-\eta/2,\, \alpha+\beta-1/6, \, 5\beta/2-1/2). 
\] 
 Thus, it follows from~\eqref{eq:asympD} and~\eqref{eq:defquad} that
 \[
 \mathcal{D}_m(H) = m\left(1+\frac{\log p}{\log H}\right) + O \left( mp^{\omega+o(1)}\right).
 \]
 By our hypothesis, $\omega<0$. Hence, we have the asymptotic formulas~\eqref{eq: C and D} again 
and we get 
\begin{equation}\label{Cauchy-even}\sum_{\substack{\chi \in \cX_{p,m}^+\\ L(1/2,\chi) \neq 0}} 1  \geq \frac{1}{1+1/\beta} \sum_{\chi \in \cX_{p,m}^+}  1.
\end{equation} 
The same method shows that~\eqref{Cauchy-even} holds when summing over odd characters.

 \section*{Acknowledgement}

During the preparation of this work, M.M. was supported by the Ministero della Istruzione e della Ricerca 
Young Researchers Program Rita Levi Montalcini and I.S.  by the  
Australian Research Council Grants DP230100530 and DP230100534.

\end{document}